\author{\Large{Damanvir Singh Binner}}
\documentclass[11pt]{article}
\usepackage[margin=1in]{geometry}
\usepackage{amsfonts,amsmath,amssymb, amsthm}
\usepackage{relsize}
\begin{document}
\theoremstyle{plain} 
\newtheorem{theorem}{Theorem}
 \newtheorem{corollary}[theorem]{Corollary} 
 \newtheorem{lemma}[theorem]{Lemma} 
 \newtheorem{proposition}[theorem]{Proposition}
 
\theoremstyle{definition}
\newtheorem{definition}[theorem]{Definition}
\newtheorem{example}[theorem]{Example}
\newtheorem{conjecture}[theorem]{Conjecture}
\theoremstyle{remark}
\newtheorem{remark}[theorem]{Remark}

\title{\Large{Number of Non-negative Integer Solutions of the Equation $ax+by+cz=n$  and its Relation with Quadratic Residues}}
\date{}
\maketitle
\begin{center}
\vspace*{-8mm}
\large{Department of Mathematics, Simon Fraser University\\
8888 University Dr, Burnaby, BC V5A 1S6 \\
dbinner@sfu.ca}
\end{center}
\begin{center}
\large{Abstract}
\end{center}
 In 2000, A. Tripathi $\cite{AT}$ used generating functions to obtain a formula for the number of non-negative solutions $(x,y)$ of the equation $ax + by = n$ where $a$, $b$ and $n$ are given positive integers. We generalize this procedure for the number of solutions of the equation $ax + by + cz = n$. The formula leads us to a surprising connection between the number of solutions of the equation $ax + by + cz = n$ and quadratic residues. As a consequence of our work, we are able to prove the equivalence between two fundamental results by Gauss and Sylvester in the nineteenth century which are generally viewed to be independent. 
 
\section {Introduction}
\label{Intro}

Let $a$,$b$,$c$ and $n$ be given positive integers. The purpose of this note is to calculate the number of solutions $N(a,b,c;n)$ of the equation $ax + by + cz = n$ in non-negative integer tuples $(x,y,z)$. Note that if $\gcd(a,b,c)$ does not divide $n$, then the equation cannot have any solutions and if it does divide $n$, then we can divide both sides of the equation by this common factor. Thus without loss of generality, we can assume that $\gcd (a,b,c) = 1$. We will first show that there is also no loss of generality in making the assumption that $a$, $b$ and $c$ are pairwise coprime. This will allow us to use generating functions to find an explicit formula for the number of solutions. In Section $\ref{Equi}$ of this note, we establish the equivalence between two well-known results of Gauss and Sylvester which are as follows:

\begin{theorem}[Gauss (1808)]
\label{Quad}
For distinct odd primes $p$ and $q$, $$\sum_{i=1}^{\frac{p-1}{2}}\Big\lfloor \frac{iq}{p} \Big\rfloor + \sum_{i=1}^{\frac{q-1}{2}}\Big\lfloor \frac{ip}{q}\Big\rfloor  =  \frac{(p-1)(q-1)}{4}, $$
\end{theorem}

\begin{theorem}[Sylvester (1882)]
\label{Sylvester}
 If $p$ and $q$ are distinct odd prime numbers, the number of natural numbers which cannot be expressed in the form $px + qy$  for non-negative integers $x$ and $y$ is equal to $\frac{(p-1)(q-1)}{2}$.
\end{theorem}

Theorem $\ref{Quad}$ was proved in $\cite{Gauss}$. Gauss used it to give his third proof of the law of quadratic reciprocity. Eisentein $\cite{Eisenstein}$ gave a geometric proof of Theorem $\ref{Quad}$ in 1844. For more information about these classical proofs, see $\cite{Classical}$. Theorem $\ref{Sylvester}$ is a special case of a result proved in $\cite{Sylvester82}$. In the general case, $p$ and $q$ just need to be coprime natural numbers instead of distinct odd primes. An easy proof can be found in $\cite{Sylvester}$.

Throughout this note, $\lfloor x \rfloor$ denotes the greatest integer less than or equal to $x$.

\section {The Main Theorem}
\label{Mainthm}

\subsection{Reduction to pairwise coprime case}
\label{Reduct}

Let $a$, $b$, $c$ and $n$ be positive integers and as justified above, assume $\gcd(a,b,c) = 1$. We define the following symbols:

\begin{itemize}
\item Let $g_1$, $g_2$ and $g_3$ denote $\gcd(b, c)$,  $\gcd(c, a)$ and $\gcd(a, b)$ respectively. Note that $\gcd(g_1, g_2) = \gcd(g_2, g_3) = \gcd(g_3, g_1) = 1$. 
\item Let $a_1$, $b_2$ and $c_3$ denote the modular inverse of $a$ with respect to the modulus $g_1$,  $b$  with respect to the modulus $g_2$ and $c$ with respect to the modulus $g_3$ respectively.
\item Let $n_1$, $n_2$ and $n_3$ denote the remainder upon dividing $na_1$ by $g_1$, $nb_2$ by $g_2$ and $nc_3$ by $g_3$ respectively.
\item Let $A =  \frac{a}{g_2 g_3}$,  $B =  \frac{b}{g_3 g_1}$ and $C =   \frac{c}{g_1 g_2}$.
\item Define $N =  \frac{n - an_1 - bn_2 - cn_3}{g_1 g_2 g_3}$. Note that $N$ is an integer.
\end{itemize}

\begin{lemma}
\label{Reduction}
With the above notation, the number of solutions of the equation $ax + by + cz = n$ in non-negative integer tuples $(x,y,z)$ is equal to the number of solutions of the equation $Ax + By + Cz = N$.
\end{lemma}

\begin{proof}
Let $S$ and $T$  denote the solution sets of $ax + by + cz  = n$ and $Ax + By + Cz = N$ respectively. Then the function $\phi : S\rightarrow T$ such that $$(x,y,z)\mapsto \Big(\frac{x-n_1}{g_1},   \frac{y-n_2}{g_2}, \frac{z-n_3}{g_3} \Big)$$ provides the required bijection.
\end{proof}

Since $A$, $B$ and $C$ are pairwise coprime positive integers, Lemma $\ref{Reduction}$ shows that there is no loss of generality in making the assumption that $a$, $b$ and $c$ are pairwise coprime. 

\subsection{Statement of Theorem and Proof}
\label{StatProof}
As justified above, we assume $\gcd(a,b) = \gcd(b,c) = \gcd(c,a) =1$. We define a few other symbols:

\begin{itemize}
	\item Define $ b'_1$ such that  $b'_1 \equiv - nb^{-1}$ (mod $a$)  with $ 1\leq b'_1 \leq a$. Moreover define $ c'_1$ such that $c'_1 \equiv bc^{-1}$ (mod $a$)  with $ 1\leq c'_1 \leq a$.
	\item Define  $ c'_2$ such that  $c'_2 \equiv - nc^{-1}$ (mod $b$)  with $ 1\leq c'_2 \leq b$. Moreover define $ a'_2$ such that $a'_2 \equiv ca^{-1}$ (mod $b$)  with $ 1\leq a'_2\leq b$.
	\item Define $ a'_3$ such that $a'_3 \equiv - na^{-1}$ (mod $c$)  with $ 1\leq a'_3 \leq c$. Moreover define  $b'_3$ such that $b'_3 \equiv ab^{-1}$ (mod $c$)  with $ 1\leq b'_3 \leq c$.
	\item $N_1 = n(n + a + b +c) + cbb'_1(a+1-c'_1(b'_1-1)) + acc'_2(b+1 - a'_2(c'_2-1))$ $+ baa'_3 (c+1-b'_3(a'_3-1)) $.
\end{itemize}

\begin{theorem}
\label{MainThm}
With the notation above, the number of solutions of the equation $ax + by +cz = n$ with $gcd (a,b) = gcd (b,c) = gcd (c,a) = 1$ is given by 
$$N(a,b,c;n) =   \frac{N_1}{2abc} + \sum_{i=1}^{b'_1-1} \Big\lfloor \frac{ic'_1}{a} \Big\rfloor + \sum_{i=1}^{c'_2-1} \Big\lfloor \frac{ia'_2}{b} \Big\rfloor + \sum_{i=1}^{a'_3-1} \Big\lfloor \frac{ib'_3}{c} \Big\rfloor  - 2 .$$
 \end{theorem}
 
 \begin{proof}
By elementary combinatorics, we know that the number of solutions of $ax + by + cz = n$ is equal to the coefficient of $x^n$ in $$\frac{1}{(1-x^a)(1-x^b)(1-x^c)}.$$  Let $\zeta_m$ denote $e^{{\frac{2\pi i}{m}}}$.  We know that
$$(1-x^a)(1-x^b)(1-x^c)  =  (1-x)^3  \prod_{k = 1}^{a-1}(1-\zeta_a^{-k}x) \prod_{k = 1}^{b-1}(1-\zeta_b^{-k}x)  \prod_{k = 1}^{c-1}(1-\zeta_c^{-k}x). $$ Note that because $a$, $b$ and $c$ are pairwise coprime, $1-\zeta_a^{-k}x$, $ 1-\zeta_b^{-k}x$ and $ 1-\zeta_c^{-k}x$ are distinct for all values of $k$. Thus we obtain the following partial fraction decomposition:
     \begin{equation}
     \label{Partial Fraction}
  \frac{1}{(1-x^a)(1-x^b)(1-x^c)} = \frac{c_1}{1-x} + \frac{c_2}{(1-x)^2} +\frac{c_3}{(1-x)^3} + \sum_{k=1}^{a-1}\frac{A_k}{1-\zeta_a^{-k} x} + \sum_{k=1}^{b-1}\frac{B_k}{1-\zeta_b^{-k} x}  + \sum_{k=1}^{c-1}\frac{C_k}{1-\zeta_c^{-k} x}. 
     \end{equation}
   On comparing the coefficients of $x^n$ on both sides of $\eqref{Partial Fraction}$, we find
   \begin{equation}
\label{Coefficient}
    N(a,b,c;n) =  c_1 + (n + 1) c_2 + \frac{(n + 2)(n + 1)}{2} c_3  + \sum_{k=1}^{a-1}A_k   \zeta_a^{-nk} + \sum_{k=1}^{b-1}B_k   \zeta_b^{-nk} +  \sum_{k=1}^{c-1}C_k   \zeta_c^{-nk}. 
   \end{equation}
   If we substitute $x=0$ in $\eqref{Partial Fraction}$, we get 
   \begin{equation}
      \label{x=0}
    1 = c_1 + c_2 + c_3 + \sum_{k=1}^{a-1} A_k + \sum_{k=1}^{b-1} B_k + \sum_{k=1}^{c-1} C_k.
   \end{equation}
  Upon subtracting $\eqref{x=0}$ from $\eqref{Coefficient}$ to cancel $c_1$ and thus simplify our calculation, we get
   \begin{equation}
   \label{Almost}
   N(a,b,c;n)  -  1 =  nc_2+ \frac{n(n+3)}{2} c_3 - \sum_{k=1}^{a-1}A_k(1-\zeta_a^{-nk}) - \sum_{k=1}^{b-1}B_k(1-\zeta_b^{-nk}) - \sum_{k=1}^{c-1}C_k  (1-\zeta_c^{-nk}).
    \end{equation}
   The usual procedure for finding coefficients of partial fraction gives $$c_3 = \frac{1}{abc},  c_2 = \frac{a+b+c-3}{2abc},$$ $$A_k = \frac{1}{a(1-\zeta_a^{bk})(1-\zeta_a^{ck})}, B_k = \frac{1}{b(1-\zeta_b^{ck})(1-\zeta_b^{ak})}   \textrm{ and }   C_k = \frac{1}{c(1-\zeta_c^{ak})(1-\zeta_c^{bk})}.$$ Substituting these back in $\eqref{Almost}$, we have 
   \begin{equation}
   \label{Formula}
   N(a,b,c;n) =  \frac{n(n+a+b+c)}{2abc} + 1 - \left(\frac{S_1}{a} + \frac{S_2}{b} + \frac{S_3}{c}\right),
   \end{equation}
  where $$ S_1 =  \sum_{k=1}^{a-1} \frac{1-\zeta_a^{-nk}}{(1-\zeta_a^{bk})(1-\zeta_a^{ck})}, S_2 = \sum_{k=1}^{b-1} \frac{1-\zeta_b^{-nk}}{(1-\zeta_b^{ck})(1-\zeta_b^{ak})} \textrm{ and } S_3 = \sum_{k=1}^{c-1} \frac{1-\zeta_c^{-nk}}{(1-\zeta_c^{ak})(1-\zeta_c^{bk})}.$$ Equations similar to $\eqref{Formula}$ have been known before in the pairwise coprime case (see, for example $\cite{komatsu}$). They do not solve these sums explicitly but express them as summations of some complicated sums, products and quotients of trigonometric functions which could be solved for small values of $a$, $b$ and $c$. However, the summations become intractable as $a$, $b$ or $c$ become slightly large. In this paper, we would express the number of solutions in terms of summations of floor functions which are quite easy to work with, particularly with the help of Lemma $\ref{GenReci}$ that we would describe in the next section.
  
  Next, we find $S_1$, $S_2$ and $S_3$. By definition of $b'_1$, we have $bb'_1 \equiv -n $ (mod $a$), so $\zeta_a^{-nk} = \zeta_a^{bb'_1k}$, and thus
  \begin{equation}
  \label{sums}
    S_1 =  \sum_{k=1}^{a-1} \frac{1-\zeta_a^{bb'_1k}}{(1-\zeta_a^{bk})(1-\zeta_a^{ck})} = \sum_{k=1}^{a-1} \sum_{j = 0}^{b_1^{'}-1} \frac{\zeta_a^{jbk}}{1-\zeta_a^{ck}} = \sum_{k=1}^{a-1} \sum_{j = 0}^{b_1^{'}-1} \frac{1}{1-\zeta_a^{ck}} - \sum_{k=1}^{a-1} \sum_{j = 0}^{b_1^{'}-1} \frac{1-\zeta_a^{jbk}}{1-\zeta_a^{ck}}.  
    \end{equation}
 It is well known that $$ \sum_{k=1}^{a-1} \frac{1}{1-\zeta_a^{ck}} = \frac{a-1}{2}, $$
    and thus changing the order of summations yields
  \begin{equation}
  \label{first}
   \sum_{k=1}^{a-1} \sum_{j = 0}^{b_1^{'}-1} \frac{1}{1-\zeta_a^{ck}} = b'_1 \Big(\frac{a-1}{2}\Big).  
   \end{equation} 
  By definition of $c'_1$, we have $cc'_1 \equiv b$ (mod $a$), so $\zeta_a^{jbk} = \zeta_a^{jcc'_1k}$ and thus
  \begin{equation}
  \label{second}
    \sum_{k=1}^{a-1}\sum_{j = 0}^{b'_1-1} \frac{1 - \zeta_a^{jbk}}{1-\zeta_a^{ck}} = \sum_{k=1}^{a-1}\sum_{j = 1}^{b'_1-1} \frac{1 - \zeta_a^{jbk}}{1-\zeta_a^{ck}} = \sum_{k=1}^{a-1}\sum_{j = 1}^{b'_1-1} \frac{1 - \zeta_a^{jcc'_1k}}{1-\zeta_a^{ck}} =  \sum_{k=1}^{a-1}\sum_{j = 1}^{b'_1-1}\sum_{l=0}^{jc'_1-1}\zeta_a^{lck}. 
   \end{equation}
From $\eqref{sums}$, $\eqref{first}$ and $\eqref{second}$, we get that
\begin{equation}
\label{final}
 S_1 = b'_1 \Big(\frac{a-1}{2}\Big) - \sum_{k=1}^{a-1}\sum_{j = 1}^{b'_1-1}\sum_{l=0}^{jc'_1-1}\zeta_a^{lck}. 
\end{equation}
Now, 
\begin{equation}
\label{Triple}
 \sum_{k=1}^{a-1}\sum_{j = 1}^{b'_1-1}\sum_{l=0}^{jc'_1-1}\zeta_a^{lck} =  \sum_{j = 1}^{b'_1-1}\sum_{l=0}^{jc'_1-1} \sum_{k=1}^{a-1} \zeta_a^{lck} = \sum_{j = 1}^{b'_1-1}\sum_{l=0}^{jc'_1-1} \sum_{k=0}^{a-1} \zeta_a^{lck} - \frac{c'_1b'_1(b'_1-1)}{2}. 
 \end{equation}
We know that $ \sum_{k=0}^{a-1} \zeta_a^{lck} \neq 0$ only if $a$ divides $l$ in which case the sum is $a$. Note that here we have again used the fact that $\gcd(a,c) = 1$. Therefore, $$ \sum_{l=0}^{jc'_1-1}  \sum_{k=0}^{a-1} \zeta_a^{lck} = a \Big(\Big\lfloor \frac{jc'_1-1}{a}  \Big\rfloor +1 \Big) = a\Big( \Big\lfloor \frac{jc'_1}{a}  \Big\rfloor +1 \Big). $$ In the last step we have used that $\gcd(a,c'_1) = 1$ and that $j \leq b'_1 - 1 \leq a-1$. Hence, 
\begin{equation}
\label{Floor}
\sum_{j = 1}^{b'_1-1}\sum_{l=0}^{jc'_1-1} \sum_{k=0}^{a-1} \zeta_a^{lck} =  a\sum_{j = 1}^{b'_1-1} \Big(\Big\lfloor \frac{jc'_1}{a}  \Big\rfloor +1\Big). 
\end{equation}
From $\eqref{final}$, $\eqref{Triple}$ and $\eqref{Floor}$, we have $$ \frac{S_1}{a}  =     b'_1\Big(\frac{a-1}{2a}\Big) + \frac{c'_1b_1^{'}(b'_1-1)}{2a} - \sum_{j=1}^{b'_1-1}\Big\lfloor \frac{jc'_1}{a} \Big\rfloor - (b'_1-1). $$ We combine the first and the last terms to get 
\begin{equation}
\label{S1}
\frac{S_1}{a}  =  \frac{c'_1b_1^{'}(b'_1-1)}{2a} - \sum_{j=1}^{b'_1-1}\Big\lfloor \frac{jc'_1}{a} \Big\rfloor + 1 - b'_1\Big(\frac{a+1}{2a}\Big).
\end{equation}
Symmetrically, we also have 
 \begin{equation}
\label{S2}
\frac{S_2}{b}= \frac{a'_2c'_2(c'_2-1)}{2b} - \sum_{j=1}^{c'_2-1}\Big\lfloor \frac{ja'_2}{b}\Big\rfloor  + 1 -  c'_2\Big(\frac{b+1}{2b}\Big), 
\end{equation}
and
\begin{equation}
\label{S3}
\frac{S_3}{c}  =  \frac{b'_3a'_3(a'_3-1)}{2c} - \sum_{j=1}^{a'_3-1}\Big\lfloor \frac{jb'_3}{c} \Big\rfloor +1 - a'_3\Big(\frac{c+1}{2c}\Big).
\end{equation}
The result now follows from $\eqref{Formula}$, $\eqref{S1}$, $\eqref{S2}$ and $\eqref{S3}$.
\end{proof}
\subsection{An algorithm to find these sums}
\label{Algo}
To be able to calculate the number of solutions faster, we need better methods to calculate these summations involving floor functions. Let us recall Theorem $\ref{Quad}$ which states that for distinct odd primes $p$ and $q$, $$\sum_{i=1}^{\frac{p-1}{2}}\Big\lfloor \frac{iq}{p} \Big\rfloor + \sum_{i=1}^{\frac{q-1}{2}}\Big\lfloor \frac{ip}{q}\Big\rfloor  =  \frac{(p-1)(q-1)}{4}. $$ It turns out that we can generalize this result for arbitrary numerator and denominator.
  \begin{lemma}
  \label{GenReci}
  Let $a$, $b$, $c$ and $K$ be positive integers such that $b < a$, $c < a$, $\gcd (a,c)  = 1$ and $K = \Big\lfloor \frac{bc}{a} \Big\rfloor$. Then  $$\sum_{i=1}^{b} \Big\lfloor \frac{ic}{a} \Big\rfloor + \sum_{i=1}^{K} \Big\lfloor \frac{ia}{c} \Big\rfloor = bK.$$ 
  \end{lemma}
  \begin{proof}
   We have  $$\sum_{i=1}^{b}\Big\lfloor \frac{ic}{a} \Big\rfloor =  \sum_{t=1}^{K}t n_t, $$ where $n_t$ is the number of $i$ such that $1 \leq i \leq b$ and $\lfloor \frac{ic}{a} \rfloor = t$.                                   
 Clearly $n_t = \Big\lfloor \frac{(t+1)a}{c} \Big\rfloor - \Big\lfloor \frac{ta}{c} \Big\rfloor $ if $ t < K$ and $n_K  = b - \Big\lfloor \frac{Ka}{c}\Big\rfloor$. Therefore, $$\sum_{i=1}^{b}\Big\lfloor \frac{ic}{a} \Big\rfloor = \sum_{t=1}^{K-1} \Big(\Big\lfloor\frac{(t+1)a}{c}\Big\rfloor - \Big\lfloor \frac{ta}{c} \Big\rfloor \Big) t + \Big(b-\Big\lfloor \frac{Ka}{c}\Big\rfloor \Big)K.$$ We rearrange the terms and solve the summation using telescoping sums to obtain $$ \sum_{i=1}^{b} \Big \lfloor \frac{ic}{a} \Big\rfloor = \sum_{t=1}^{K-1} \Big(\Big\lfloor \frac{(t+1)a}{c}\Big\rfloor (t+1) - \Big\lfloor \frac{ta}{c} \Big\rfloor t \Big) - \sum_{t=1}^{K-1}\Big\lfloor \frac{(t+1)a}{c} \Big\rfloor + bK - K \Big \lfloor \frac{Ka}{c} \Big\rfloor.$$
 By canceling terms and solving, we get the required result.
     \end{proof}
     
Lemma $\ref{GenReci}$ is helpful to calculate summations of the form $\sum_{i=1}^{b} \Big\lfloor \frac{ic}{a} \Big\rfloor$ because it reduces to a summation of the same form with a lower upper limit of summation and a lower denominator. We shall see in Section $\ref{Efficiency}$ that after two applications of this lemma, the upper limit of summation and the denominator both reduce to less than half while the numerator is still less than the denominator.
     
\begin{remark}
Observe that if we take $a = p$, $b = \frac{p-1}{2}$ and $c = q$ in Lemma $\ref{GenReci}$,  then we get Theorem $\ref{Quad}$. Similar to Eisenstein's proof of Theorem $\ref{Quad}$ (see $\cite{Eisenstein}$ or $\cite{Classical}$ for details), we can also give a geometric proof of Lemma $\ref{GenReci}$ by counting the number of points under the straight line $y=\frac{c}{a} x$.
\end{remark}

Let us describe the algorithm for finding the number of solutions $N(a,b,c;n)$ of the equation $ax+by+cz = n$.

\begin{enumerate}

\item Reduce the given equation to an equation with $\gcd(a,b,c) = 1$ as described in Section $\ref{Intro}$. Then reduce it to the pairwise coprime case as described in Section $\ref{Reduct}$.

\item Apply the formula in Theorem $\ref{MainThm}$ to get the number of solutions in terms of the three summations involving floor functions.

\item Suppose the first summation looks like $\sum_{i=1}^{b_1} \Big\lfloor \frac{ic_1}{a_1} \Big\rfloor $ for some positive integers  $a_1$,$b_1$ and $c_1$ such that $b _1< a_1 $, $c_1 < a_1$. Then we apply Lemma $\ref{GenReci}$ to get the summation in terms of the summation $\sum_{i=1}^{K_1} \Big\lfloor \frac{ia_1}{c_1} \Big\rfloor$, where $K _1= \Big\lfloor \frac{b_1c_1}{a_1} \Big\rfloor$.

\item Now we have the summation $\sum_{i=1}^{K_1} \Big\lfloor \frac{ia_1}{c_1} \Big\rfloor$. Note that we cannot apply Lemma $\ref{GenReci}$ again to find this sum since $a_1 > c_1 $. However, by division algorithm, we have $a_1=c_1q+r$ for some quotient $q$ and remainder $r$. Then $\sum_{i=1}^{K} \Big\lfloor \frac{ia_1}{c_1} \Big\rfloor = \frac{qK(K+1)}{2} + \sum_{i=1}^{K} \Big\lfloor \frac{ir}{c_1} \Big\rfloor $. Since $r < c_1$, we can use Lemma $\ref{GenReci}$  again to find this sum. 

\item Keep repeating Steps $4$ and $5$ till the first summation in Step $3$ is fully solved. Then follow the same procedure to find the other two summations and hence the number of solutions.
 
 \end{enumerate}
 
\subsection{An example}
\label{Example}

Let us apply this algorithm to an example. Consider the equation $$4452x + 8030y + 9945z = 3857942.$$ Let $N$ denote the number of solutions of this equation. Note that $\gcd(4452,8030) = 2$, $\gcd(4452,9945) = 3$ and $\gcd(8030,9945) = 5$. If we apply the reduction to coprime case as in Lemma $\ref{Reduction}$, we get that the number of solutions of the above equation is equal to the number of solutions of the equation $$ 742 x + 803 y + 663 z = 128598. $$  Now we apply the formula of Theorem $\ref{MainThm}$ to get 
\begin{equation}
\label{Eqn 16}
N = \sum_{i=1}^{129} \Big\lfloor \frac{281i}{742} \Big\rfloor + \sum_{i=1}^{539} \Big\lfloor \frac{621i}{803} \Big\rfloor + \sum_{i=1}^{335} \Big\lfloor \frac{602i}{663} \Big\rfloor -166300.
\end{equation}
In order to solve the first sum, we apply Lemma $\ref{GenReci}$ to get 
\begin{equation}
\label{Eqn 17}
 \sum_{i=1}^{129} \Big\lfloor \frac{281i}{742} \Big\rfloor = 6192 -  \sum_{i=1}^{48} \Big\lfloor \frac{742i}{281} \Big\rfloor.   
\end{equation}
Now,
\begin{equation}
\label{Eqn 18}
\sum_{i=1}^{48} \Big\lfloor \frac{742i}{281} \Big\rfloor = \sum_{i=1}^{48} \Big(2i + \Big\lfloor \frac{180i}{281} \Big\rfloor \Big) = 2352 + \sum_{i=1}^{48} \Big\lfloor \frac{180i}{281} \Big\rfloor.  
\end{equation}
Another application of Lemma $\ref{GenReci}$  gives 
\begin{equation}
\label{Eqn 19}
 \sum_{i=1}^{48} \Big\lfloor \frac{180i}{281} \Big\rfloor = 1440 -   \sum_{i=1}^{30} \Big\lfloor \frac{281i}{180} \Big\rfloor = 975 -  \sum_{i=1}^{30} \Big\lfloor \frac{101i}{180} \Big\rfloor.  
\end{equation}
Repeating the above procedure, we have 
\begin{equation}
\label{Eqn 20}
  \sum_{i=1}^{30} \Big\lfloor \frac{101i}{180} \Big\rfloor = 480 -   \sum_{i=1}^{16} \Big\lfloor \frac{180i}{101} \Big\rfloor  = 344 -   \sum_{i=1}^{16} \Big\lfloor \frac{79i}{101} \Big\rfloor,  
\end{equation}
\begin{equation}
\label{Eqn 21}
   \sum_{i=1}^{16}   \Big\lfloor \frac{79i}{101} \Big\rfloor  =  192 -  \sum_{i=1}^{12}   \Big\lfloor \frac{101i}{79} \Big\rfloor  = 114 -  \sum_{i=1}^{12}   \Big\lfloor \frac{22i}{79} \Big\rfloor,       
 \end{equation}
 \begin{equation}
 \label{Eqn 22}
        \sum_{i=1}^{12} \Big\lfloor \frac{22i}{79} \Big\rfloor  =  36 -  \sum_{i=1}^{3} \Big\lfloor \frac{79i}{22} \Big\rfloor = 18 -  \sum_{i=1}^{3} \Big\lfloor \frac{13i}{22} \Big\rfloor,
      \end{equation}
and
      \begin{equation}
      \label{Eqn 23}
       \sum_{i=1}^{3} \Big\lfloor \frac{13i}{22} \Big\rfloor = 3 - \sum_{i=1}^{1} \Big\lfloor \frac{22i}{13} \Big\rfloor = 3-1 =2.
\end{equation}
From $\eqref{Eqn 17}$ to $\eqref{Eqn 23}$, we get $$\sum_{i=1}^{129} \Big\lfloor \frac{281i}{742} \Big\rfloor = 3111.$$ Repeating the same procedure with the other two summations leads to $$\sum_{i=1}^{539}  \Big\lfloor \frac{621i}{803} \Big\rfloor = 112277,$$ and $$\sum_{i=1}^{335} \Big\lfloor \frac{602i}{663} \Big\rfloor = 50934. $$ Substituting these values back in $\eqref{Eqn 16}$, we find $N = 22$, i.e., there are $22$ solutions for the equation $4452x + 8030y + 9945z = 3857942 $ in non-negative tuples $(x,y,z)$.
 
 \subsection{Efficiency of the algorithm}
 \label{Efficiency}
We want to find an upper bound for the number of steps required to calculate the number of solutions of the equation $ax+by+cz = n$. Suppose we want to find the sum $\sum_{i=1}^b \Big\lfloor \frac{ic_1}{a_1} \Big\rfloor$ for some positive integers $a_1$, $b$ and $c_1$ such that $b < a_1, c_1 < a_1$ and  $\gcd(c_1,a_1) = 1$. 

According to Step $4$ of the algorithm, we need to apply Lemma $2$ to get $\sum_{i=1}^b \Big\lfloor \frac{ic_1}{a_1} \Big\rfloor$ in terms of the sum $\sum_{i=1}^{K_1} \Big\lfloor \frac{ia_1}{c_1} \Big\rfloor$ for some $K_1 < c_1$. Then as step $5$ in the algorithm describes, we need to apply the division algorithm $a_1 = c_1q_1+a_2$ where $a_2 < c_1$.  Since $\gcd(c_1,a_1) =1$, so $\gcd(a_2,c_1) =1$. Note that since $c_1 < a_1$, so $q \geq 1$ and thus $$a_1 \geq c_1+a_2 > 2a_2,$$ or equivalently $a_2 < \frac{a_1}{2}$. With the help of this division algorithm, the sum  $\sum_{i=1}^{K_1} \Big\lfloor \frac{ia_1}{c_1} \Big\rfloor$ can be obtained in terms  of the sum  $\sum_{i=1}^{K_1} \Big\lfloor \frac{ia_2}{c_1} \Big\rfloor$.

 According to step $5$ of the algorithm, we again apply Lemma $\ref{GenReci}$ to get the sum  $\sum_{i=1}^{K_1} \Big\lfloor \frac{ia_2}{c_1} \Big\rfloor$ in terms of the sum  $\sum_{i=1}^{K_2} \Big\lfloor \frac{ic_1}{a_2} \Big\rfloor$ for some $K_2 < a_2$. Then we again apply division algorithm $c_1 = a_2q_2+c_2$ for some $c_2 < a_2$ to get the sum $\sum_{i=1}^{K_2} \Big\lfloor \frac{ic_1}{a_2} \Big\rfloor$ in terms of the sum $\sum_{i=1}^{K_2} \Big\lfloor \frac{ic_2}{a_2} \Big\rfloor$. Since $\gcd(a_2,c_1) =1$, so $\gcd(c_2,a_2) =1$. Finally, since $K_2 < a_2$, $c_2 < a_2$ and $\gcd(c_2,a_2)=1$, we return to Step $4$ of the algorithm to find the sum $\sum_{i=1}^{K_2} \Big\lfloor \frac{ic_2}{a_2} \Big\rfloor$. 

Thus with one application of the Steps $4$ and $5$ of the algorithm (i.e. two applications of both Lemma $\ref{GenReci}$ and the division algorithm), we can obtain the sum $\sum_{i=1}^b \Big\lfloor \frac{ic_1}{a_1} \Big\rfloor$ in terms of the sum $\sum_{i=1}^{K_2} \Big\lfloor \frac{ic_2}{a_2} \Big\rfloor$ where $a_2 < \frac{a_1}{2}$. It is also easy to see that $K_2 < \frac{b}{2}$. This makes sure that the Steps $4$ and $5$ of the algorithm terminate in $O(\log a)$ steps and hence the algorithm terminates in $O(\log t)$ steps where $t = \max(a,b,c)$. 

\subsection{Relationship with quadratic residues}
\label{Relation}
 Let us recall the Eisenstein's Lemma which states that for given odd distinct primes $p$ and $q$, the quadratic residue is given by $\Big(\frac{q}{p}\Big) = (-1)^ t $, where $t =  \sum_{i=1}^{\frac{p-1}{2}}\Big\lfloor \frac{i q}{p} \Big\rfloor$. 
 
 Thus the quadratic residues are related to some summations with which we have been dealing while attempting to solve the equation $ax+by+cz = n$. This suggests that we might be able to find an equation whose number of solutions gives the quadratic residues. 

\begin{lemma}
\label{EqnSum}
The number of solutions of the equation $px + qy + z =  \frac{q(p-1)}{2}$ is equal to $$N_{p,q} = \frac{p-1}{2}+\sum_{i=1}^{ \frac{p-1}{2}}\Big\lfloor \frac{iq}{p} \Big\rfloor.$$ So, $\Big(\frac{q}{p} \Big) = (-1)^{N_{p,q}-\frac{p-1}{2}}$.
\end{lemma}

\begin{proof}
Clearly, one way of proving this is by applying Theorem $\ref{MainThm}$. However we could also prove it directly by fixing a $y$ and then calculating the number of possibilities for $x$. For a given $x$ and $y$, $z$ is automatically determined. 
\end{proof} 

\section{Equivalence between two well-known results}
\label{Equi}
The aim of this section is to establish the equivalence between Theorem  $\ref{Quad}$ and Theorem $\ref{Sylvester}$. Throughout this section, $p$ and $q$ would denote distinct odd primes.
\begin{lemma}
\label{Threetotwo}
The number of solutions of the equation $px + qy + z =  \frac{p(q-1)}{2} + \frac{q(p-1)}{2}$ is equal to  $\frac{p(q-1)}{2} + \frac{q(p-1)}{2}+1 - N_0$
where $N_0$ is the number of natural numbers which cannot be expressed in the form $px + qy$  such that  $x$ and $y$ are non-negative integers.
\end{lemma} 
\begin{proof}
We first fix a $z$ and then look for the number of solutions of the equation $$px + qy  =  \frac{p(q-1)}{2} + \frac{q(p-1)}{2} - z. $$ Thus the number of solutions of the equation $px + qy + z = \frac{p(q-1)}{2} + \frac{q(p-1)}{2}$ is equal to 
\begin{equation}
\label{ThreeTwo}
  \sum_{n=0}^{\frac{p(q-1)}{2} + \frac{q(p-1)}{2}}S_n,  
  \end{equation}
  where $S_n$ is the number of solutions of the equation $px + qy  =  n$. Clearly $S_0 = 1$. Furthermore, it is well known that whenever $ 1 \leq n \leq (p-1)(q-1)$, $S_n$ could either be $0$ or $1$ and whenever $ (p-1)(q-1) < n < pq$, $S_n = 1$.  We also prove these facts in Section $\ref{Two}$ using the methods developed in Section $\ref{Reduction}$. Thus by the definition of $N_0$,  $$ \sum_{n=1}^{(p-1)(q-1)}S_n = (p-1)(q-1) - N_0,$$ and, $$ \sum_{n=(p-1)(q-1)+1}^{\frac{p(q-1)}{2} + \frac{q(p-1)}{2}} S_n  = \frac{p(q-1)}{2} + \frac{q(p-1)}{2} - (p-1)(q-1). $$ Therefore, $$ \sum_{n=0}^{\frac{p(q-1)}{2} + \frac{q(p-1)}{2}}S_n = S_0 +  \sum_{n=0}^{(p-1)(q-1)}S_n + \sum_{n=(p-1)(q-1)+1}^{\frac{p(q-1)}{2} + \frac{q(p-1)}{2}} S_n =  \frac{p(q-1)}{2} + \frac{q(p-1)}{2}+1 - N_0. $$
\end{proof}
We calculate the number of solutions of the equation $px + qy + z =  \frac{p(q-1)}{2} + \frac{q(p-1)}{2}$ in another way by considering four separate cases. We recall that $N_{p,q}$ denotes the number of solutions of the equation $px+qy+z =  \frac{q(p-1)}{2}$. 
\begin{lemma}
\label{SolnOther}
The number of solutions of the equation $px + qy + z =  \frac{p(q-1)}{2} + \frac{q(p-1)}{2}$ is equal to
$2(N_{p,q} + N_{q,p}) - \Big(\frac{p+1}{2} +  \frac{q+1}{2} + 1 \Big)$.
\end{lemma}

\begin{proof}
Let $X$, $Y$ and $Z$ denote $  \frac{q-1}{2}-x$, $ \frac{p-1}{2}-y$ and $ \frac{q(p-1)}{2}-z$ respectively. We split our calculation into four different cases according to: 
\begin{enumerate}
\item $ X \geq 0, Y \geq 0, Z \geq 0.$
\item $ X \geq 0, Y \geq 0, Z < 0.  $
\item $ X \geq 0, Y < 0. $
\item $ X < 0. $
\end{enumerate}

Case $1$ : Let $$ X \geq 0, Y \geq 0, Z \geq 0.$$ Let $S_1$ denote the set of solutions of $px + qy + z =  \frac{p(q-1)}{2} + \frac{q(p-1)}{2}$ in Case $1$, and let  $T_1$ denote the set of solutions of $px + qy + z =  \frac{q(p-1)}{2}$. Then the function $\phi_1 : S_1 \rightarrow T_1$ such that $$(x, y, z)\mapsto (X,Y,Z) $$ is a bijection. Thus the number of solutions of $px + qy + z =  \frac{p(q-1)}{2} + \frac{q(p-1)}{2}$ in Case $1$ is equal to $N_{p,q}$.

Case $2$ :     Let  $$ X \geq 0, Y \geq 0, Z < 0.  $$ Let $S_2$ denote the set of solutions of $ px + qy + z =  \frac{p(q-1)}{2} + \frac{q(p-1)}{2}$ such that $ X \geq 0$, $Y \geq 0$ and $Z \leq 0$, and let $T_2$ denote the set of solutions of $px + qy + z =  \frac{p(q-1)}{2}$. Then the function $\phi_2 : S_2 \rightarrow T_2$ such that $$ (x, y, z)\mapsto (x, y,-Z)$$ is a bijection. However, to calculate the number of solutions of $px + qy + z =  \frac{p(q-1)}{2} + \frac{q(p-1)}{2}$ in Case $2$, we need to subtract the solutions in which $Z=0$, but if $Z=0$, then the equation becomes $px + qy  =  \frac{p(q-1)}{2}$ which has a unique solution $x = \frac{q-1}{2}$ and $y =0$ because $p$ divides $y$  and $y < p$ together imply $y = 0$. Thus the number of solutions of $px + qy + z =  \frac{p(q-1)}{2} + \frac{q(p-1)}{2}$ in Case $2$ is equal to $N_{q,p} - 1$.

Case $3$ : Let  $$ X \geq 0, Y < 0. $$ Let $S_3$ denote the set of solutions of $ px + qy + z =  \frac{p(q-1)}{2} + \frac{q(p-1)}{2}$  such that $ X \geq 0$ and $Y \leq 0 $ , and let $T_3$ denote the set of solutions of $px + qy + z =  \frac{p(q-1)}{2}$. Then the function $\phi_3 : S_3 \rightarrow T_3$ such that $$(x, y, z)\mapsto (x,-Y, z)$$ is a bijection. However, to calculate the number of solutions of the equation in Case $3$, one has to subtract the number of solutions with $Y=0$, but if $Y=0$, then the equation becomes $px + z  =  \frac{p(q-1)}{2}$ which clearly has $\frac{q+1}{2}$ possible values for $x$ and for a given $x$, $z$  is automatically determined. Thus the number of solutions of $px + qy + z =  \frac{p(q-1)}{2} + \frac{q(p-1)}{2}$ in Case $3$ is equal to $N_{q,p} - \frac{q+1}{2}$.

Case $4$ :  Let $$ X < 0. $$ Let $S_4$ denote the set of solutions of $px + qy + z =  \frac{p(q-1)}{2} + \frac{q(p-1)}{2}$ such that $ X \leq 0 $, and let $T_4$ denote the set of solutions of $px + qy + z =  \frac{q(p-1)}{2}$. Then the function $\phi_4 : S_4 \rightarrow T_4$ such that $$ (x, y, z) \mapsto (-X, y, z)$$ is a bijection. However, to calculate the number of solutions of the equation in Case $4$, one has to subtract the number of solutions with $X=0$, but if $X=0$, then the equation becomes $qy + z  =  \frac{q(p-1)}{2}$ which clearly has $\frac{p+1}{2}$ solutions. Thus the number of solutions of $px + qy + z =  \frac{p(q-1)}{2} + \frac{q(p-1)}{2}$ in Case $4$ is equal to $N_{p,q} - \frac{p+1}{2}$.

 Adding up the solutions in all the above cases, we get the required result.
 \end{proof}
 
Upon comparing the number of solutions of the equation $px + qy + z =  \frac{p(q-1)}{2} + \frac{q(p-1)}{2}$ obtained in Lemma $\ref{Threetotwo}$ and Lemma $\ref{SolnOther}$, and then using Lemma $\ref{EqnSum}$, we get $$ N_0 + 2 \Big(\sum_{i=1}^{\frac{p-1}{2}} \Big\lfloor \frac{iq}{p} \Big\rfloor + \sum_{i=1}^{\frac{q-1}{2}} \Big\lfloor \frac{ip}{q} \Big\rfloor \Big)  =  (p-1)(q-1). $$ This establishes the required equivalence between Theorem $\ref{Quad}$ and Theorem $\ref{Sylvester}$.

\section{Some applications of techniques developed in this paper}
\label{Applications}
\subsection{Another proof of Theorem $\ref{Quad}$}

In this section, we prove Theorem $\ref{Quad}$ by counting the number of solutions of an equation in two different ways. Without loss of generality, we can assume $q < p$. Recall that in Lemma $\ref{EqnSum}$, we counted the number of solutions of the equation $ px + qy + z =  \frac{q(p-1)}{2}$. Now we will count these in another way. 
\begin{lemma}
\label{Other}
If $p$ and $q$ are distinct odd primes such that $q < p$, then the number of solutions of the equation $ px + qy + z =  \frac{q(p-1)}{2}$, with $q < p$, is given by $$ N_{p,q} = \frac{p+1}{2} +  \frac{(p-1)(q-1)}{4} - \sum_{i=1}^ {\frac{q-1}{2}} \Big\lfloor \frac{ip}{q} \Big\rfloor.  $$
\end{lemma}
\begin{proof}
Maximum possible value for $x$ is $$ \Big\lfloor \frac{q(p-1)}{2p} \Big\rfloor =  \Big\lfloor \frac{(q-1)}{2} + \frac{p-q}{2p}\Big\rfloor = \frac{q-1}{2}.$$  Now we consider two cases :

Case $1$ : Let $x = 0$. Number of solutions in Case $1$ is equal to $\frac{p+1}{2}$.

Case $2$ :  Let $x \geq 1$. Fix $x = i$. Then the number of possible values for $y$ is equal to $$ 1 + \Big\lfloor \frac{\frac{q(p-1)}{2}-ip}{q} \Big\rfloor = \frac{p-1}{2} - \Big\lfloor \frac{ip}{q} \Big\rfloor. $$ Hence the total number of solutions in Case $2$ is equal to $$ \sum_{i=1}^ {\frac{q-1}{2}} \Big(\frac{p-1}{2} - \Big\lfloor \frac{ip}{q} \Big\rfloor \Big) =   \frac{(p-1)(q-1)}{4} - \sum_{i=1}^ {\frac{q-1}{2}} \Big\lfloor \frac{ip}{q} \Big\rfloor.   $$ Adding up the number of solutions in both the cases, we get the required result.
\end{proof}

Theorem $\ref{Quad}$ now easily follows from Lemma $\ref{EqnSum}$ and Lemma $\ref{Other}$.

\subsection{Solving the equation $ax+by = n$}
\label{Two}
If we modify the techniques used in Section $2.1$ to reduce the equation $ax+by+cz = n$ to the pairwise coprime case, we can completely solve the equation $ax+by = n$. Note that if $\gcd(a,b)$ does not divide $n$, then there is no solution and if it does, then we can divide out both sides of the equation by $\gcd(a,b)$. Thus without loss of generality, we can assume $\gcd(a,b) =1$. 

Let $a^{-1}$ and $b^{-1}$ denote the modular inverse of $a$ with respect to $b$ and $b$ with respect to $a$ respectively. Further, let $a_1$ and $b_1$ denote the remainder when $na^{-1}$ is divided by $b$ and $nb^{-1}$ is divided by $a$ respectively. 

\begin{theorem}
\label{2vari}
Let $a$ and $b$ be coprime positive integers and $n$ is some positive integer. Then the number of solutions of the equation $ax+by=n$ is given by $$N(a,b;n) = 1+ \frac{n-aa_1-bb_1}{ab} $$
\end{theorem}
\begin{proof}
Let $S$ and $T$ denote the solution sets of $ax+by=n$ and $x+y = \frac{n-aa_1-bb_1}{ab} $ respectively. Then the function $\phi : S \rightarrow T$ such that $$ (x,y) \mapsto \Big(\frac{x-a_1}{b},\frac{y-b_1}{a}\Big) $$ is a bijection. 
\end{proof}

This formula is equivalent to the one given in $\cite{AT}$.
\begin{corollary}
\label{TwoVari}
$ax+by=n$ has a unique solution if $(a-1)(b-1) \leq n < ab$.
\end{corollary}
\begin{proof}
If $n < ab$, then clearly $N(a,b;n) < 2$. Moreover, since $a_1 \leq (b-1)$ and $b_1 \leq (a-1)$, so $N(a,b;n) \geq \frac{n+a+b-ab}{ab}$. Therefore, if $(a-1)(b-1) \leq n$, then $N(a,b;n) > 0$. Thus whenever $(a-1)(b-1) \leq n < ab$, $N(a,b;n) = 1$. 
\end{proof}

Recall that the Frobenius number of a set $\{a_1,a_2,\ldots, a_l\}$ such that $\gcd(a_1,a_2,\ldots, a_l) = 1$ is defined as the largest integer which cannot be expressed in the form $$ k_1a_1 + k_2a_2 + \cdots + k_la_l $$ where $k_1,k_2, \ldots, k_l$ are non-negative integers. 

Corollary $\ref{TwoVari}$ gives another proof of the fact that the Frobenius number of the set $\{a,b\}$ such that $\gcd(a,b) = 1$ is equal to $(a-1)(b-1)$. For more information about Frobenius numbers and the formula for three variables, see $\cite{AT2}$.

\subsection{A by-product summation result}
\label{By-Prod}
 We can modify the proof of Lemma $\ref{Other}$ to obtain the following result :
\begin{theorem}
Let $p$ and $q$ be distinct odd primes such that $ p < q$, then $$\mathlarger\sum_{i=\frac{q-1}{2}- \Big\lfloor \frac{q-p}{2p} \Big\rfloor}^{\frac{q-1}{2}} \Big\lfloor \frac{ip}{q} \Big\rfloor =  \Big(\frac{p-1}{2}\Big) \Big(\Big\lfloor \frac{q-p}{2p} \Big\rfloor+1\Big).$$
\end{theorem}

For example, if we take $p=23$ and $q = 739$, then $ \Big\lfloor \frac{q-p}{2p} \Big\rfloor=15$, and we get $$ \sum_{i=354}^{369} \Big\lfloor \frac{23i}{739} \Big\rfloor = 176.$$

\subsection{An application of Theorem $\ref{Quad}$}
\label{Appli}
We want to show that Theorem $\ref{Quad}$ can be used to determine the parity of the number of solutions of a particular linear equation. Let $p$ and $q$ be distinct odd primes, and let $p^{-1}$ denote the modular inverse of $p$ with respect to $q$. Further, let $k$ denote the quantity $\Big(\frac{p-1}{2}\Big) +p\Big(\frac{q-1}{2}\Big)p^{-1}$.
\begin{theorem}
The number of solutions of the equation $$px+qy+z=k$$ has the same parity as $$ (k+1)\Big(\frac{k+p+q}{2}\Big) + \Big(\frac{q^2-1}{8}\Big)(1+p^{-1}) + \frac{(p-1)(q-1)}{4}. $$
\end{theorem}

\begin{proof}
By Eisenstein's Lemma, $\Big(\frac{p}{q}\Big) = (-1)^{t_1}$ and $\Big(\frac{p^{-1}}{q}\Big) = (-1)^{t_2}$ where $t_1 = \sum_{i=1}^{\frac{q-1}{2}} \Big\lfloor \frac{ip}{q} \Big\rfloor$ and $t_2 = \sum_{i=1}^{\frac{q-1}{2}} \Big\lfloor  \frac{ip^{-1}}{q} \Big\rfloor$. Note that $$\Big(\frac{p}{q}\Big) = \Big(\frac{p^{-1}}{q}\Big). $$ Therefore,  $ \sum_{i=1}^{\frac{q-1}{2}} \Big\lfloor \frac{ip}{q} \Big\rfloor$ and $\sum_{i=1}^{\frac{q-1}{2}} \Big\lfloor  \frac{ip^{-1}}{q} \Big\rfloor$ have the same parity. The result now follows directly from Theorem $\ref{MainThm}$ and Theorem $\ref{Quad}$. 
\end{proof}

\section{Acknowledgements}
\label{Ackn}
I am highly grateful to A. Tripathi at IIT Delhi who asked me to try to extend his work for $3$ variables. Section $\ref{StatProof}$ of this note heavily uses the techniques developed in his paper. I also want to thank my friend A. Jain for verifying the reduction process by running a computer program and also for simplifying the expression of $N$ in Section $\ref{Reduct}$.  I am highly indebted to A. Rattan at SFU Burnaby, J. O. Shallit at the University of Waterloo and A. Ganguli at IISER Mohali for helping me with the presentation of this paper. I also wish to express my gratitude to D. Prasad at TIFR, Mumbai for listening to my work very patiently and giving me encouraging remarks.

\bibliographystyle{plain}
\bibliography{ax+by+cz=n_and_Quadratic_Residues}

\end{document}